\documentclass[reqno]{amsart}
\usepackage{amsmath}
\usepackage{amssymb}
\usepackage{graphicx}
\usepackage{amsfonts}

\newtheorem{theorem}{Theorem}
\theoremstyle{plain}

\newtheorem{claim}{Claim}

\newtheorem{example}{Example}

\newtheorem{lemma}{Lemma}

\newtheorem{proposition}{Proposition}
\newtheorem{remark}{Remark}

\numberwithin{equation}{section}

\begin{document}
\title[Fractal Networks modelled on Sierpinski Gasket]{Asymptotic formula on average path length
of fractal networks modelled on Sierpinski Gasket}
\author{Fei Gao}
\address{School of Computer Science and Information Technology, Zhejiang
Wanli University, 315100 Ningbo, P. R. China}
\email{dgaofei@sina.com}
\author{Anbo Le}
\address{Institute of Mathematics, Zhejiang Wanli University, 315100 Ningbo,
P. R. China}
\email{anbole@msn.com}
\author{Lifeng Xi}
\address{Institute of Mathematics, Zhejiang Wanli University, 315100 Ningbo,
P. R. China}
\email{xilifengningbo@yahoo.com}
\author{Shuhua Yin}
\address{Institute of Mathematics, Zhejiang Wanli University, 315100 Ningbo,
P. R. China}
\email{yinshuhua@126.com}
\thanks{Lifeng Xi is the corresponding author. The work is supported by NSFC
(Nos.\ 11371329, 11471124), NSF of Zhejiang(Nos.\ LR13A1010001, LY12F02011).}

\begin{abstract}
In this paper, we introduce a new method to construct evolving networks
based on the construction of the Sierpinski gasket. Using self-similarity
and renewal theorem, we obtain the asymptotic formula for average path
length of our evolving networks.
\end{abstract}

\maketitle

\section{Introduction}

The Sierpinski gasket described in 1915 by W.~Sierpi\'{n}ski is a classical
fractal. Suppose $K$ is the solid regular triangle with vertexes $%
a_{1}=(0,0),$ $a_{2}=(1,0),$ $a_{3}=(1/2,\sqrt{3}/2).$ Let $%
T_{i}(x)=x/2+a_{i}/2$ be the contracting similitude for $i=1,2,3.$ Then $%
T_{i}:K\rightarrow K\ $and the Sierpinski gasket $E$ is the self-similar
set, which is the unique invariant set \cite{H} of IFS $\{T_{1},T_{2},T_{3}%
\},$ satisfying%
\begin{equation*}
E=\cup _{i=1}^{3}T_{i}(E).
\end{equation*}%
The Sierpinski gasket is important for the study of fractals, e.g., the Sierpinski gasket is a typical example of post-critically finite self-similar fractals on which the Dirichlet forms and Laplacians can be constructed by Kigami \cite{Ki1,Ki2},
see also Strichartz \cite{St}.

\begin{figure}[tbph]
\centering\includegraphics[width=0.65\textwidth]{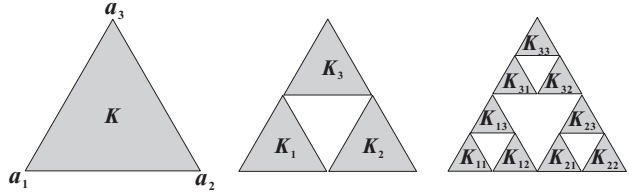} \vspace{-0.35cm}
\caption{The first two constructions of Sierpinski gasket}
\end{figure}

For the word $\sigma =i_{1}\cdots i_{k}$ with letters in $\{1,2,3\}$, i.e.,
every letter $i_{t}\in \{1,2,3\}$ for all $t\leq k,$ we denote by $|\sigma
|(=k)$ the length of word $\sigma .$ Given words $\sigma =i_{1}\cdots i_{k}$
and $\tau =j_{1}\cdots j_{n},$ we call $\sigma $ a prefix of $\tau $ and
denote by $\tau \prec \sigma ,$ if $k<n$ and $i_{1}\cdots i_{k}=j_{1}\cdots
j_{k}.$ We also write $\tau \preceq \sigma $ if $\tau =\sigma $ or $\tau
\prec \sigma .$ When $\tau \prec \sigma $ with $|\tau |=|\sigma |-1,$ we say
that $\tau $ is the father of $\sigma $ and $\sigma $ is a child of $\tau .$
Given $\sigma =i_{1}\cdots i_{k},$ we write $T_{\sigma }=T_{i_{1}}\circ
\cdots \circ T_{i_{k}}\text{ and }K_{\sigma }=T_{\sigma }(K)$ which is a
solid regular triangle with side length $2^{-|\sigma |}.$ For notational
convenience, we write $K_{\emptyset }=K$ with empty word $\emptyset .$ We
also denote $|\emptyset |=0.$ If $\tau \prec \sigma ,$ then $K_{\sigma
}\subset K_{\tau }.$ For solid triangle $K_{\sigma }$ with word $\sigma ,$
we denote by $\partial K_{\sigma }$ its boundary consisting of $3$ sides,
where every side is a line segment with side length $2^{-|\sigma |}.$

Complex networks arise from natural and social phenomena, such as the
Internet, the collaborations in research, and the social relationships.
These networks have in common two structural characteristics: the
small-world effect and the scale-freeness (\emph{power-law} degree
distribution), as indicated, respectively, in the seminal papers by Watts
and Strogatz \cite{WS98} and by Barab\'{a}si and Albert \cite{BA99}. In fact
complex networks also exhibit \emph{self-similarity} as demonstrated by
Song, Havlin and Makse \cite{SHM05} and fractals possess the feature of
\emph{power law} in terms of their fractal dimension (e.g. see \cite%
{Falconer03}). Recently self-similar fractals are used to model evolving
networks, for example, in a series of papers, Zhang et al.\ \cite%
{ZZFGZ07,ZZZCG08,GWZZW09} use the Sierpinski gasket to construct evolving
networks. There are also some complex networks modelled on self-similar
fractals, for example, Liu and Kong \cite{LK10} and Chen et al.\ \cite{CFW12}
study Koch networks, Zhang et al.\ \cite{ZZCYG07} investigate the
networks constructed from Vicsek fractals. See also Dai and Liu
\cite{D}, Sun et al. \cite{S} and Zhou et al. \cite{Zhou}.

\medskip

In the paper, we introduce a new method to construct evolving networks
modelled on Sierpinski gasket and study the asymptotic formula for average
path length.

Since $E$ is connected, we can construct the network from geometry as
follows.

Fix an integer $t,$ we consider a network $G_{t}$ with vertex set $%
V_{t}=\{\sigma :0\leq |\sigma |\leq t\}$ where $\#V_{t}=1+3+...+3^{t}=\frac{1%
}{2}(3^{t+1}-1).$ For the edge set of $G_{t},$ there is a unique edge
between distinct words $\sigma $ and $\tau $ (denoted by $\tau \sim \sigma $%
) if and only if
\begin{equation}
\partial K_{\sigma }\cap \partial K_{\tau }\neq \varnothing .  \label{vvv}
\end{equation}%
We can illustrate the geodesic paths in Figure 2 for $t=3$.
We have $233\sim 32\sim 312$ since $\partial K_{233}\cap \partial
K_{32}=\{C\}$ and $\partial K_{32}\cap \partial K_{312}=\{F\}.$ We also get
another geodesic path from $233\ $to $312:$ $233\sim 3\sim 312$ since $%
\partial K_{233}\cap \partial K_{3}=\{C\}$ and $\partial K_{3}\cap \partial
K_{312}=[G,F],$ the line segment between $G$ and $F.$ We also have some
geodesic paths from $21\ $to $312:$ $21\sim 1\sim 311\sim 312,\ 21\sim
\emptyset \sim 3\sim 312$ and $21\sim 2\sim 32\sim 312.$ Also we have $%
132\sim 1\sim \emptyset $ but $132\not\sim \emptyset ,$ then the geodesic
distance between $132$ and $\emptyset $ is $2.$

\begin{figure}[tbph]
\centering\includegraphics[width=0.9\textwidth]{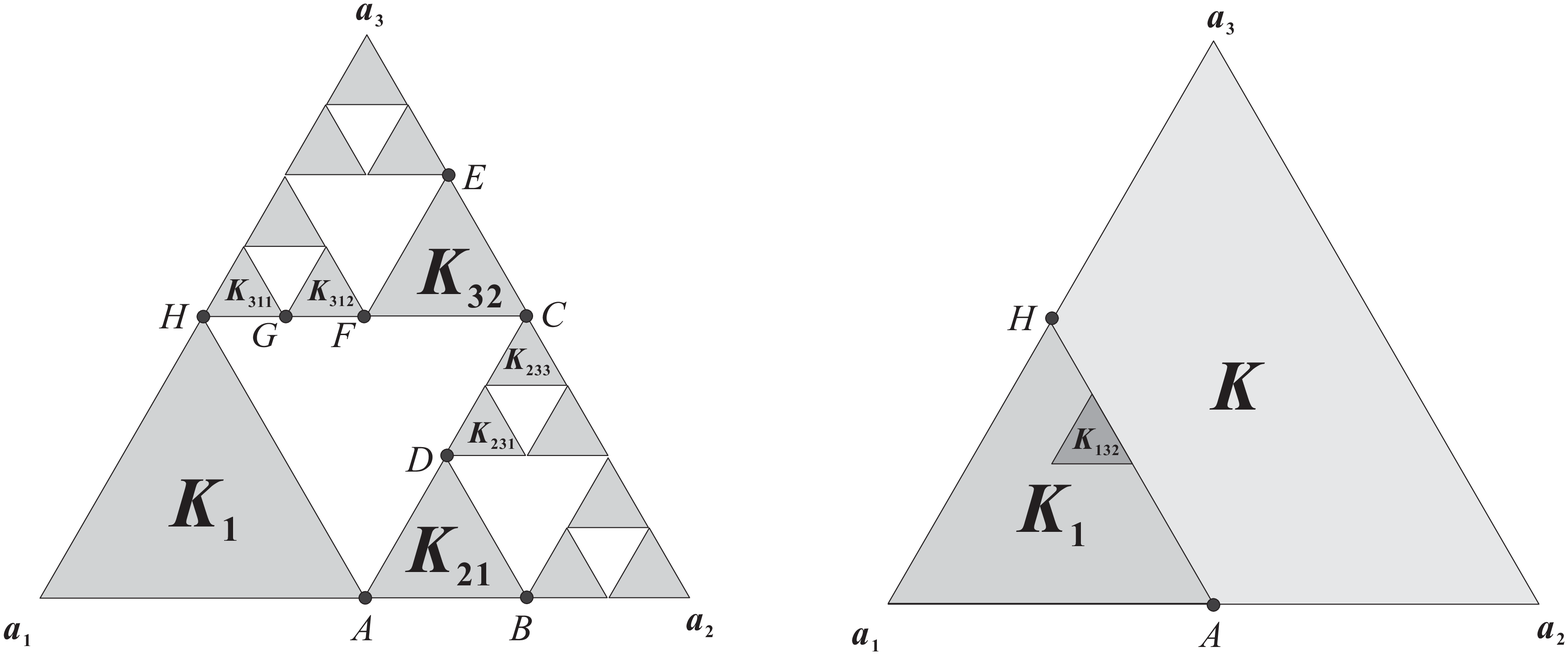} \vspace{-0.35cm}
\caption{{}}
\end{figure}

In Figure 2, $312\not\sim \emptyset $, $231\not\sim \emptyset $ and $132\not\sim
\emptyset$. In fact, by observation we have

\begin{claim}
\label{C:two}Suppose $\sigma \prec \tau $ with $\tau =\sigma \beta .$ Then $%
\sigma \sim \tau $ if and only if there are at most two letters in $\beta .$
In particular, if $\sigma \prec \tau $ and $|\tau |-|\sigma |\leq 2,$ then $%
\sigma \sim \tau .$
\end{claim}

For example, $123123$ and $1231$ are neighbors, but $123123$ and $123$ are
not.

For every $t,$ we denote $d_{t}(\sigma ,\tau )$ the geodesic distance on $%
V_{t}.$ Let%
\begin{equation*}
\bar{D}(t)=\frac{\sum\nolimits_{\sigma \neq \tau \in V_{t}}d_{t}(\sigma
,\tau )}{\#V_{t}(\#V_{t}-1)/2}
\end{equation*}%
be the average path length of the complex network $V_{t}.$

We can state our main result as follows.

\begin{theorem}
We have the asymptotic formula
\begin{equation}
\lim_{t\rightarrow \infty }\frac{\bar{D}(t)}{t}=\frac{4}{9}.  \label{a0}
\end{equation}
\end{theorem}

\begin{remark}
\label{R:6} Since $t\propto \ln (\#V_{t})$, Theorem $1$ implies that the
evolving networks $G_{t}$ have small average path length, namely $\bar{D}%
(t)\propto \ln (\#V_{t})$.
\end{remark}

%We can draw $G_{5}$ by using Pajek in Figure 3.

%\begin{figure}[tbph]
%\centering\includegraphics[width=0.5\textwidth]{6.eps} \vspace{-0.35cm}
%\caption{$G_{5}$}
%\end{figure}

The paper is organized as follows. In Section 2 we give notations and sketch of proof
for Theorem 1,consisting of four steps. In sections 3-6, we will provide details for
the four steps respectively. Our main techniques come from the
self-similarity and the renewal theorem.

\section{Sketch of proof for Theorem 1}

We will illustrate our following four steps needed to prove Theorem 1.

\textbf{Step 1.} We calculate the geodesic distance between a word and the
empty word.

Given a small solid triangle $\Delta ,$ we can find a maximal solid triangle
$\Delta ^{\prime }$ which contains $\Delta $ and their boundaries are
touching. Translating into the language of words, for a given word $\sigma
\neq \emptyset ,$ we can find a unique shortest word $f(\sigma )$ such that $%
f(\sigma )\prec \sigma $ and $f(\sigma )\sim \sigma .$ For a word $\sigma
=\tau _{2}\tau _{1},$ where $\tau _{1}$ is the maximal suffix with at
most two letters appearing, using Claim \ref{C:two} we have $f(\sigma )=\tau _{2}.$
Iterating $f$ again and again, we obtain a sequence $\sigma \sim f(\sigma
)\sim \cdots \sim f^{n-1}(\sigma )\sim f^{n}(\sigma )=\emptyset .$ Let
\begin{equation*}
\omega (\sigma )=\min \{n:f^{n}(\sigma )=\emptyset \}.
\end{equation*}%
In particular, we define $\omega (\emptyset )=0.$ For $\sigma
=112113112312=(112)(11311)(23)(12),$ we have $f(\sigma )=(112)(11311)(23),$ $%
f^{2}(\sigma )=(112)(11311)$, $f^{3}(\sigma )=(112)\ $and $f^{4}(\sigma
)=\emptyset .$ Then $\omega (\sigma )=4.$ We will prove in Section 3

\begin{proposition}
For $\sigma \in V_{t},$ we have $d_{t}(\sigma ,\emptyset )=\omega (\sigma ).$
\end{proposition}

In fact, this proposition shows that $d_{t}(\sigma ,\emptyset )$ is
independent of the choice of $t$ whenever $t\geq |\sigma |.$ In this case, we
also write $d(\sigma ,\emptyset ).$ Write
\begin{equation*}
L(\tau )=d(\tau ,\emptyset )-1\text{ for }\tau \neq \emptyset \text{ and }%
L(\emptyset )=0.
\end{equation*}%
Then $L(\tau )$ is independent of $t,$ in fact, $L(\tau )$ is the minimal
number of moves for $K_{\tau }$ to touch the boundary of $K.$

\medskip

\textbf{Step 2.} Given $m\geq 1,$ we consider the average geodesic
distance between the empty word and word of length $m$ and set
\begin{equation*}
\,\bar{\alpha}_{m}=\frac{\sum\nolimits_{|\sigma |=m}d(\sigma ,\emptyset )}{%
\#\{\sigma :|\sigma |=m\}}-1=\frac{\sum\nolimits_{|\sigma |=m}L(\sigma )}{%
\#\{\sigma :|\sigma |=m\}},
\end{equation*}%
and $\bar{\alpha}_{0}=0,$ we will obtain the limit property of $\bar{\alpha}%
_{m}/m$ as $m\rightarrow \infty$ in Section 4.

In fact, by the \textbf{Jordan curve theorem},\textbf{\ }we can obtain
\begin{equation}
L(\tau )+L(\sigma )\leq L(\tau \sigma )\leq L(\tau )+L(\sigma )+1.
\end{equation}%
From (2.1), we can verify $\{\bar{\alpha}_{m}\}_{m}$ is
superadditive which implies

\begin{proposition}
$\lim\limits_{m\rightarrow \infty }\bar{\alpha}_{m}/m$ $=\sup (\bar{\alpha}%
_{m}/m)<\infty .$
\end{proposition}

Denote
\begin{equation*}
\alpha ^{\ast }=\sup (\bar{\alpha}_{m}/m)=\lim\limits_{m\rightarrow \infty }%
\bar{\alpha}_{m}/m.
\end{equation*}

\medskip

\textbf{Step 3.} We obtain the asymptotic formula of $\bar{D}(t)$ in
terms of $\alpha ^{\ast }.$

Using the similarity of Sierpinski gasket, e.g., for $i=1,2,3,$%
\begin{equation*}
d(i\sigma ,i\tau )=d(\sigma ,\tau ),
\end{equation*}%
and for $i\neq j,$%
\begin{equation*}
L(\sigma )+L(\tau )\leq d(i\sigma ,j\tau )\leq (L(\sigma )+1)+(L(\tau )+1)+1,
\end{equation*}%
we will prove the following in Section 5

\begin{proposition}
\label{P:asym}$\lim\limits_{t\rightarrow \infty }\frac{\bar{D}(t)}{t}%
=2\alpha ^{\ast }.$
\end{proposition}

\begin{figure}[tbph]
\centering\includegraphics[width=0.5\textwidth]{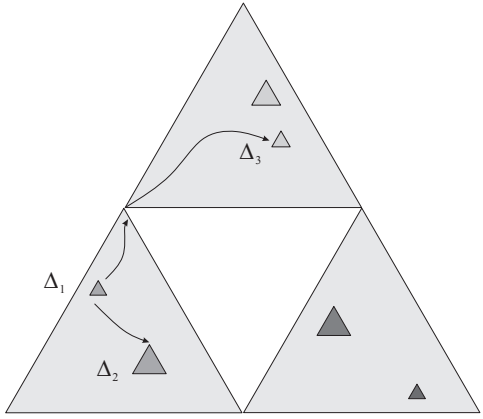} \vspace{-0.35cm}
\caption{The typical geodesic path between $\Delta _{1}$ and $\Delta _{3}$}
\end{figure}

As illustrated in Figure 3, Proposition \ref{P:asym} shows that the \emph{typical}
geodesic path is the geodesic path between $\Delta _{1}$ and $\Delta
_{3}$ whose first letters of codings are different. On the other hand, for
example the geodesic path between $\Delta _{1}$ and $\Delta _{2}$ with the same first letter
will give negligible contribution to $%
\bar{D}(t).$ Using $L(\sigma )+L(\tau )\leq d(i\sigma ,j\tau )\leq L(\sigma
)+L(\tau )+3$ with $i\neq j,$ we obtain that $d(i\sigma ,j\tau )\approx
L(\sigma )+L(\tau )$, ignoring the terms like $d(i\sigma ,i\tau ),$ we have
\begin{equation*}
\frac{\bar{D}(t)}{t}\approx 2\cdot \frac{1}{t}\cdot \frac{%
\sum\nolimits_{|\tau |\leq t-1}L(\tau )}{\#\{\tau :|\tau |\leq t-1\}},
\end{equation*}%
where $\frac{\sum\nolimits_{|\tau |\leq t-1}L(\tau )}{\#\{\tau :|\tau |\leq
t-1\}}$ is the average value of $L(\tau ).$ Using Stolz theorem, we have
\begin{eqnarray}
\lim\limits_{t\rightarrow \infty }\frac{2\sum\nolimits_{|\tau |\leq
t-1}L(\tau )}{t\cdot \#\{\tau :|\tau |\leq t-1\}} &=&\lim\limits_{t%
\rightarrow \infty }\frac{2\sum\nolimits_{|\tau |=t-1}L(\tau )}{3^{t-1}t+%
\frac{3^{t}}{6}-\frac{1}{2}}  \notag \\
&=&\lim\limits_{t\rightarrow \infty }\frac{2}{t}\cdot \frac{%
\sum\nolimits_{|\tau |=t-1}L(\tau )}{\#\{\tau :|\tau |=t-1\}}  \label{mmm} \\
&=&\lim\limits_{t\rightarrow \infty }\frac{2\bar{\alpha}_{t-1}}{t}=2\alpha
^{\ast }.  \notag
\end{eqnarray}

\bigskip

\textbf{Step 4.} Using the renewal theorem,
we will prove in Section 6

\begin{proposition}
$\alpha ^{\ast }=2/9.$
\end{proposition}

By programming, we have
\begin{equation*}
\begin{tabular}{|l|l|l|l|l|l|l|}
\hline
$t=$ & $300$ & $400$ & $500$ & $600$ & $700$ & $800$ \\ \hline
$\bar{\alpha}_{t}/t=$ & $0.2207\cdots $ & $0.2211\cdots $ & $0.2213\cdots $
& $0.2214\cdots $ & $0.2215\cdots $ & $0.2216\cdots $ \\ \hline
\end{tabular}%
\end{equation*}%
which is in line with $\alpha ^{\ast }=2/9=0.2222\cdots.$

In fact, suppose $\Sigma =\{\cdots x_{2}x_{1}:$ $x_{i}=1,2$ or $3$ for all $%
t\}$ is composed of infinite words with letters in $\{1,2,3\}.$ Then we have
a natural mass distribution $\mu $ on $\Sigma $ such that for any word $%
\sigma $\ of length $k,$
\begin{equation*}
\mu (\{\cdots x_{k}\cdots x_{1}:x_{k}\cdots x_{1}=\sigma \})=1/3^{k}.
\end{equation*}%
For any word $\sigma ,$ let $\#(\sigma )$ denote the cardinality of letters
appearing in word $\sigma .$ For $\mu $-almost all $x,$ let
\begin{equation*}
S(\cdots x_{p}x_{p-1}\cdots x_{1})=p\text{ if }\#(x_{p-1}\cdots x_{1}1)=2%
\text{ and }\#(x_{p}x_{p-1}\cdots x_{1}1)=3.
\end{equation*}%
Then $\mathbb{E}(S)=\sum_{k=2}^{\infty }k\cdot \mu
\{x:S(x)=k\}=\sum_{k=2}^{\infty }k\cdot (2^{k}-2)/3^{k}=9/2<\infty .$

For $\mu $-almost all $x=\cdots x_{2}x_{1},$ suppose $x_{0}=1$ and $p_{0}=0$
and there is an infinite sequence $\{p_{n}\}_{n\geq 0}$ of integers such that
$p_{n+1}>p_{n}$ for all $n$ and
\begin{equation*}
x1=\cdots x_{p_{3}}\cdots x_{p_{2}}\cdots x_{p_{1}}\cdots x_{1}x_{0}
\end{equation*}%
satisfying $\#(x_{p_{n}}\cdots x_{p_{n-1}})=3$ and $\#(x_{(p_{n}-1)}\cdots
x_{p_{n-1}})=2$ for all $n.$ We then let
\begin{equation*}
S_{i}(x)=p_{i}-p_{i-1}\text{ for any }i\geq 1.
\end{equation*}%
Since $x_{p_{n}}$ is uniquely determined by word $x_{(p_{n}-1)}\cdots
x_{p_{n-1}},$ and $\mu $ is symmetric for letters in $\{1,2,3\},$ we find
that $\{S_{i}\}_{i}$ is a sequence of positive independent identically
distributed random variables with $S_{1}=S.$ For example, for $x=\cdots
321223121,$ then $x1=\cdots 3(21)(223)(1211)$ and $S_{1}(x)=S(x)=4,$ $%
S_{2}(x)=3$, $S_{3}(x)=2,\cdots .$

Let $J_{n}=S_{1}+\cdots +S_{n}$ and $Y_{t}=\sup \{n:J_{n}\leq t\}.$ Then $%
J_{n}=p_{n}$ and $Y_{t}=\max \{n: p_{n}\leq t\}.$ By the elementary renewal theorem, we
have
\begin{equation*}
\frac{\mathbb{E}(Y_{t})}{t}\rightarrow \frac{1}{\mathbb{E}(S)}=\frac{2}{9}%
=0.222\cdots .
\end{equation*}%
Using the following estimates (Lemma \ref{l: lim} in Section 6)
\begin{eqnarray*}
\liminf\limits_{t\rightarrow \infty }\frac{\sum\nolimits_{|\sigma
|=t}d(\sigma ,\emptyset )}{t3^{t}} &\geq &\liminf\limits_{t\rightarrow
\infty }\frac{\sum\nolimits_{k=2}^{t-1}\mathbb{E}(Y_{t-k})\frac{2^{k}-2}{%
3^{k}}}{t}, \\
\limsup_{t\rightarrow \infty }\frac{\sum\nolimits_{|\sigma |=t}d(\sigma
,\emptyset )}{t3^{t}} &\leq &\limsup_{t\rightarrow \infty }\frac{%
\sum\nolimits_{k=2}^{t-1}\mathbb{E}(Y_{t-k})\frac{2^{k}-2}{3^{k}}}{t},
\end{eqnarray*}%
we can prove that $\alpha ^{\ast }=\lim\limits_{t\rightarrow \infty }\frac{%
\sum\nolimits_{|\sigma |=t}d(\sigma ,\emptyset )}{t3^{t}}=\lim\limits_{t%
\rightarrow \infty }\frac{\mathbb{E}(Y_{t})}{t}=2/9.$

\medskip Theorem 1 follows from Propositions 3 and 4.

\bigskip

\section{Basic formulas on geodesic distance}

\subsection{\textbf{Criteria of }neighbor}

$\ $

Given distinct words $\sigma $ and $\tau $ with $|\sigma |,|\tau |\leq t,$
we give the following \textbf{criteria} to test whether they are neighbors
or not. At first, we delete the common prefix of $\sigma $ and $\tau ,$ say $%
\sigma =\beta \sigma ^{\prime }$ and $\tau =\beta \tau ^{\prime }$ where the
first letters of $\sigma ^{\prime }$ and $\tau ^{\prime }$ are different. We
can distinguish two cases: \newline
\textbf{Case 1}. If one of $\sigma ^{\prime }$ and $\tau ^{\prime }$ is the
empty word, say $\sigma ^{\prime }=\emptyset ,$ then $\sigma $ and $\tau $%
\textbf{\ }are not neighbors if and only if every letter $i\in \{1,2,3\}$
appears in the word $\tau ^{\prime }$. \newline
\textbf{Case 2}. If neither $\sigma ^{\prime }$ nor $\tau ^{\prime }$ is the
empty word, say $i,j$ the first letters of $\sigma ^{\prime }$ and $\tau
^{\prime }$ respectively, then $\sigma $ and $\tau $\textbf{\ }are neighbors
if and only if
\begin{equation*}
\sigma ^{\prime }=i[j]^{k}\text{ and }\tau ^{\prime }=j[i]^{k^{\prime }}%
\text{ with }k,k^{\prime }\geq 0.
\end{equation*}%
We say that $\sigma $ and $\tau $ are neighbors of type $1$ or $2$ according
to Case 1 or 2 respectively.

\subsection{Estimates on distance}

$\ $

The first lemma show the \emph{self-similarity} of geodesic distance.

\begin{lemma}
\label{l:3}If $\sigma =i\sigma ^{\prime }$ and $\tau =i\tau ^{\prime }$ with
$i\in \{1,2,3\},$ then
\begin{equation*}
d_{t}(\sigma ,\tau )=d_{t}(i\sigma ^{\prime },i\tau ^{\prime
})=d_{t-1}(\sigma ^{\prime },\tau ^{\prime }).
\end{equation*}%
Consequently, given any word $i_{1}\cdots i_{k},$%
\begin{equation*}
d_{t}(i_{1}\cdots i_{k}\sigma ^{\prime },i_{1}\cdots i_{k}\tau ^{\prime
})=d_{t-k}(\sigma ^{\prime },\tau ^{\prime })\text{ for all }\sigma ^{\prime
},\tau ^{\prime }.
\end{equation*}
\end{lemma}

\begin{proof}
Fix the letter $i,$ we define $g:V_{t}\rightarrow V_{t}$ by
\begin{equation*}
g(\beta )=\left\{
\begin{array}{ll}
\beta & \text{if }i\preceq \beta ,\text{ } \\
i & \text{otherwise}.%
\end{array}%
\right.
\end{equation*}%
If we give a shortest sequence $(\sigma =)\sigma ^{1}\sim \sigma ^{2}\sim
\cdots \sim \sigma ^{k}(=\tau )$ in $G_{t}$, then $\sigma =g(\sigma
^{1})\simeq g(\sigma ^{2})\simeq \cdots \simeq g(\sigma ^{k})=\tau $ is also
a sequence and all word $\{g(\sigma ^{i})\}_{i=1}^{k}$ have the same first
letter $i.$ Deleting the first letter $i$, we get a ($\simeq $)-sequence
from $\sigma ^{\prime }$ to $\tau ^{\prime }$ in $G_{t-1}.$ Hence $%
d_{t}(\sigma ,\tau )\geq d_{t-1}(\sigma ^{\prime },\tau ^{\prime }).$

On the other hand, for given ($\sim $)-sequence from $\sigma ^{\prime }$ to $%
\tau ^{\prime }$ in $G_{t-1},$ by adding the first letter $i,$ we obtain a ($%
\sim $)-sequence $\sigma $ to $\tau $ in $G_{t}$ which implies $d_{t}(\sigma
,\tau )\leq d_{t-1}(\sigma ^{\prime },\tau ^{\prime }).$ The lemma follows.
\end{proof}

Given $t,$ let $L_{t}(\sigma )=d_{t}(\sigma ,\emptyset )-1$ for $\sigma \neq
\emptyset .$ The second lemma shows that $L_{t}(\sigma )$ is independent of $%
t$ whenever $t\geq |\sigma |.$ We can write $L(\sigma ).$

\begin{lemma}
\label{l: 1}For $\sigma ,\tau \in V_{k},$ $d_{t}(\sigma ,\tau )=d_{k}(\sigma
,\tau ).$ As a result, $L_{t}(\sigma )=L_{|\sigma |}(\sigma ).$
\end{lemma}

\begin{proof}
Fix $k\leq t.$ Let $h:V_{t}\rightarrow V_{k}$ be defined by
\begin{equation*}
h(\beta )=\left\{
\begin{array}{ll}
\beta & \text{if }|\beta |\leq k,\text{ } \\
i_{1}\cdots i_{k} & \text{if }|\beta |>k\text{ and }\beta =i_{1}\cdots
i_{k}\cdots i_{|\beta |}.%
\end{array}%
\right.
\end{equation*}%
Given $\sigma ,\tau \in V_{k},$ if we give a shortest sequence $(\sigma
=)\sigma ^{1}\sim \sigma ^{2}\sim \cdots \sim \sigma ^{k}(=\tau )$ in $G_{t}$%
, by criteria of neighbor we obtain that $\sigma =h(\sigma ^{1})\simeq
h(\sigma ^{2})\simeq \cdots \simeq h(\sigma ^{k})=\tau $ is also a ($\simeq
) $-sequence and all words in $V_{k}.$ Therefore, $d_{t}(\sigma ,\tau )\geq
d_{k}(\sigma ,\tau ).$ On the other hand, any ($\sim $)-sequence from $%
\sigma $ to $\tau $ in $G_{k}$ is also a ($\sim $)-sequence in $G_{k},$ that
means $d_{t}(\sigma ,\tau )\leq d_{k}(\sigma ,\tau ).$ The lemma follows.
\end{proof}

By the \emph{Jordan curve theorem}, when a point\ in the \emph{interior}
moves to the \emph{exterior}, it must touch the \emph{boundary}. Therefore,
we have the following

\begin{lemma}
\label{l:double}For any word $\tau \sigma $ with $\tau ,\sigma \neq
\emptyset ,$ we have
\begin{equation*}
L(\tau )+L(\sigma )\leq L(\tau \sigma )\leq L(\tau )+L(\sigma )+1.
\end{equation*}
\end{lemma}

\begin{proof}
In fact, using triangle inequality and Lemmas \ref{l:3} and \ref{l: 1}, we have%
\begin{eqnarray*}
L(\tau \sigma )=d_{t}(\tau \sigma ,\emptyset )-1 &\leq &d_{t}(\tau \sigma
,\tau )+d_{t}(\tau ,\emptyset )-1 \\
&=&d_{t-|\tau |}(\sigma ,\emptyset )+d_{t}(\tau ,\emptyset )-1 \\
&=&L(\sigma )+L(\tau )+1.
\end{eqnarray*}

On the other hand, using the self-similarity in Lemma \ref{l:3}, the minimal
number of moves for $K_{\tau \sigma }$ to touch the boundary of $K_{\tau }$
is $L(\sigma ),$ and $L(\tau )$ is the minimal number of moves for $K_{\tau
} $ to touch the boundary of $K,$ by the Jordan curve theorem, there are at
least $L(\sigma )+L(\tau )$ moves for $K_{\tau \sigma }$ to touch the
boundary of $K,$ that means $L(\tau )+L(\sigma )\leq L(\tau \sigma ).$
\end{proof}

\subsection{Proof of Proposition 1}

$\ $

Suppose $\omega $ and $f$ are defined as in Section 2. By the definition of $%
f,$ if $\tau ^{\prime }\preceq \sigma ^{\prime },$ then we have $f(\tau
^{\prime })\preceq f(\sigma ^{\prime }).$ Therefore we have

\begin{claim}
\label{C:2}If $f(\sigma )\preceq f(\tau )\preceq \sigma ,$ then $%
f^{k}(\sigma )\preceq f^{k}(\tau )\preceq f^{k-1}(\sigma )$ for all $k\geq
0. $ As a result, $|\omega (\sigma )-\omega (\tau )|\leq 1.$
\end{claim}

\begin{example}
\label{E:1}Let $i=1$, $j=3$ and $\sigma =\beta 321211$ for some word $\beta
, $ we have $f(\sigma i)=\beta 3$ and $f(\sigma j)=\beta 3212.$ Then
\begin{equation*}
f(\sigma i)\preceq f(\sigma j)\preceq \sigma i,
\end{equation*}%
and thus $|\omega (\sigma i)-\omega (\sigma j)|\leq 1.$ For $\beta =132$ and
$\beta ^{\prime }=\beta 1211$ with $\beta \sim \beta ^{\prime }$, we have $%
f(\beta )\preceq f(\beta ^{\prime })\preceq \beta ,$ then $|\omega (\beta
)-\omega (\beta ^{\prime })|\leq 1.$
\end{example}

\begin{lemma}
\label{l:omega}If $\sigma \sim \tau ,$ then $\omega (\sigma )\geq \omega
(\tau )-1.$
\end{lemma}

\begin{proof}
By the criteria of neighbor, without loss of generality, we only need to
deal with three cases: (1) $\sigma \prec \tau ;$ (2) $\sigma i[j]^{p}\ $and $%
\sigma j[i]^{q}$ with $i\neq j$ and $p,q\geq 1;$ (3) $\sigma i\ $and $\sigma
j[i]^{q}$ with $i\neq j$ and $q\geq 0.$ For case (2), it is clear that
\begin{equation*}
\omega (\sigma i[j]^{p})=\omega (\sigma j[i]^{q})=\omega (\sigma ij).
\end{equation*}%
For cases (1) and (3), we only use Claim \ref{C:2} as in Example %
\ref{E:1} above.
\end{proof}

\begin{proof}[Proof of Proposition 1]
As shown in Section 2, we can find the following path from $\sigma $ to the
empty word $\emptyset :$%
\begin{equation*}
\sigma \sim f(\sigma )\sim f^{2}(\sigma )\sim \cdots \sim f^{\omega (\sigma
)}(\sigma _{k})=\emptyset .
\end{equation*}%
That means $d_{t}(\sigma ,\emptyset )\leq \omega (\sigma ).$

It suffices to show $d_{t}(\sigma ,\emptyset )\geq \omega (\sigma ).$
Suppose on the contrary, if we give a sequence
\begin{equation*}
\sigma =\sigma _{0}\sim \sigma _{1}\sim \sigma _{2}\sim \cdots \sim \sigma
_{k}=\emptyset \text{ with }k\leq \omega (\sigma )-1,
\end{equation*}%
then $\omega (\sigma _{i+1})\geq \omega (\sigma _{i})-1$ for all $i$ by
Lemma \ref{l:omega}. Therefore, $0=\omega (\emptyset )\geq \omega (\sigma
)-k>0$ which is impossible. That means $d_{t}(\sigma ,\emptyset )=\omega
(\sigma ).$
\end{proof}

\bigskip

\section{Average geodesic distance to empty word}

We first recall some notations. For every $t,$ we denote $d_{k}(\sigma ,\tau )$ the
geodesic distance on $V_{k}.$ Given $k\geq 0,\,$let $L_{k}(\emptyset )=0$
and
\begin{equation*}
L_{k}(\sigma )=d_{k}(\sigma ,\emptyset )-1\ \text{for }\sigma \in V_{k}\text{
with }\sigma \neq \emptyset .
\end{equation*}
As shown in (\ref{mmm}), for average geodesic distance $\frac{%
\sum\nolimits_{|\tau |\leq t-1}L(\tau )}{\#\{\tau :|\tau |\leq t-1\}}+1$ to
the empty word, when we estimate $\frac{1}{t}\frac{\sum\nolimits_{|\tau |\leq
t-1}L(\tau )}{\#\{\tau :|\tau |\leq t-1\}},$ it is important for us to
estimate $\bar{\alpha}_{k}/k,$ where
\begin{equation*}
\bar{\alpha}_{k}=\frac{\sum_{|\sigma |=k}L_{k}(\sigma )}{\#\{\sigma :|\sigma
|=k\}}\text{ for }k\geq 0.
\end{equation*}%

\begin{lemma}
For any $k_{1},k_{2}\geq 1,$ we have
\begin{equation}
\bar{\alpha}_{k_{1}}+\bar{\alpha}_{k_{2}}\leq \bar{\alpha}_{k_{1}+k_{2}}\leq
\bar{\alpha}_{k_{1}}+\bar{\alpha}_{k_{2}}+1.  \label{hha}
\end{equation}%
In particular, $\{\bar{\alpha}_{k}\}_{k}$ is non-decreasing, i.e.,
\begin{equation}
\bar{\alpha}_{k+1}\geq \bar{\alpha}_{k}\text{ for all }k.  \label{mono}
\end{equation}%
As a result, for any positive integers $q$ and $k$ we obtain that%
\begin{equation}
\bar{\alpha}_{k}\geq \bar{\alpha}_{q}\big[\frac{k}{q}\big]\geq \frac{\bar{\alpha}_{q}%
}{q}(k-q+1),  \label{lower}
\end{equation}
\end{lemma}

\begin{proof}
We obtain that
\begin{equation*}
\bar{\alpha}_{k_{1}+k_{2}}=\frac{\sum\nolimits_{|\tau |=k_{1}}\sum_{|\sigma
|=k_{2}}L_{k_{1}+k_{2}}(\tau \sigma )}{\#\{\tau :|\tau |=k_{1}\}\cdot
\#\{\sigma :|\sigma |=k_{2}\}}.
\end{equation*}%
If $|\tau |=k_{1}$ and $|\sigma |=k_{2},$ using Lemma \ref{l:double}, we have%
\begin{equation}
L_{k_{1}}(\tau )+L_{k_{2}}(\sigma )\leq L_{k_{1}+k_{2}}(\tau \sigma )\leq
L_{k_{1}}(\tau )+L_{k_{2}}(\sigma )+1,  \label{oki}
\end{equation}%
which implies
\begin{eqnarray*}
\bar{\alpha}_{k_{1}+k_{2}} &\geq &\frac{\sum\nolimits_{|\tau
|=k_{1}}\sum_{|\sigma |=k_{2}}\left( L_{k_{1}}(\tau )+L_{k_{2}}(\sigma
)\right) }{\#\{\tau :|\tau |=k_{1}\}\cdot \#\{\sigma :|\sigma |=k_{2}\}}=%
\bar{\alpha}_{k_{1}}+\bar{\alpha}_{k_{2}}, \\
\bar{\alpha}_{k_{1}+k_{2}} &\leq &\frac{\sum\nolimits_{|\tau
|=k_{1}}\sum_{|\sigma |=k_{2}}\left( L_{k_{1}}(\tau )+L_{k_{2}}(\sigma
)+1\right) }{\#\{\tau :|\tau |=k_{1}\}\cdot \#\{\sigma :|\sigma |=k_{2}\}}=%
\bar{\alpha}_{k_{1}}+\bar{\alpha}_{k_{2}}+1,
\end{eqnarray*}%
then (\ref{hha}) follows. In particular, we have $\bar{\alpha}_{k+1}\geq
\bar{\alpha}_{k}+\bar{\alpha}_{1}\geq \bar{\alpha}_{k}.$

Using (\ref{hha}) repeatedly, we have $\bar{\alpha}_{qm}\geq \bar{\alpha}%
_{q(m-1)}+\bar{\alpha}_{q}\geq \cdots \geq m\bar{\alpha}_{q}.$ It follows
from (\ref{mono}) that%
\begin{equation*}
\bar{\alpha}_{qm+(q-1)}\geq \bar{\alpha}_{qm+(q-2)}\geq \cdots \geq \bar{%
\alpha}_{qm+1}\geq \bar{\alpha}_{qm}\geq m\bar{\alpha}_{q}
\end{equation*}%
which implies $\bar{\alpha}_{k}\geq \bar{\alpha}_{q}[\frac{k}{q}]\geq \frac{%
\bar{\alpha}_{q}}{q}(k-q+1).$
\end{proof}

\begin{proof}[Proof of Proposition 2]
Since $\{\bar{\alpha}_{m}\}_{m}$ is superadditive, i.e., $\bar{\alpha}%
_{k_{1}+k_{2}}\geq \bar{\alpha}_{k_{1}}+\bar{\alpha}_{k_{2}}$, by Fekete's
superadditive lemma (\cite{Fe}), the limit $\lim\limits_{m\rightarrow \infty
}\frac{\bar{\alpha}_{m}}{m}$ exists and is equal to $\sup_{m}\frac{\bar{%
\alpha}_{m}}{m}$. We shall verify that $\lim\limits_{m\rightarrow \infty }%
\frac{\bar{\alpha}_{m}}{m}<+\infty .$

Fix an integer $q.$ For any $p=0,1,\cdots ,(q-1),$ using (\ref{hha}) and (%
\ref{mono}), we have $\bar{\alpha}_{qk+p}\leq \bar{\alpha}_{q(k+1)}\leq (k+1)%
\bar{\alpha}_{q}+k.$ Letting $k\rightarrow \infty ,$ we have $%
\limsup\limits_{m\rightarrow \infty }\frac{\bar{\alpha}_{m}}{m}\leq \frac{%
\bar{\alpha}_{q}}{q}+\frac{1}{q}.$
\end{proof}

Set $\alpha ^{\ast }=\lim_{m\rightarrow \infty }\frac{\bar{\alpha}_{m}}{m}%
=\sup_{m}\frac{\bar{\alpha}_{m}}{m}.$

\medskip

\section{Asymptotic formula}

Now we will investigate
\begin{equation*}
\kappa _{t}=\frac{\sum\nolimits_{\sigma \in V_{t}}L(\sigma )}{\#V_{t}}=\frac{%
\sum\nolimits_{k=0}^{t}\sum\nolimits_{|\sigma |=k}L_{k}(\sigma )}{%
\sum\nolimits_{k=0}^{t}\#\{\sigma :|\sigma |=k\}}=\frac{\sum%
\nolimits_{k=0}^{t}\bar{\alpha}_{k}3^{k}}{\sum\nolimits_{k=0}^{t}3^{k}},
\end{equation*}%
where we let $L(\emptyset )=\bar{\alpha}_{0}=\frac{\bar{\alpha}_{0}}{0}=0.$
At first, we have%
\begin{equation}
\kappa _{t}\leq \Big(\sup_{m\geq 1}\frac{\bar{\alpha}_{m}}{m}\Big)\frac{%
\sum\nolimits_{k=0}^{t}k3^{k}}{\sum\nolimits_{k=0}^{t}3^{k}}=\alpha ^{\ast }%
\frac{\sum\nolimits_{k=0}^{t}k3^{k}}{\sum\nolimits_{k=0}^{t}3^{k}}\leq
(\alpha ^{\ast }\chi (t))t,  \label{k-upper}
\end{equation}%
where
\begin{equation}
\chi (t)=\frac{\sum\nolimits_{k=0}^{t}k3^{k}}{t\sum\nolimits_{k=0}^{t}3^{k}}=%
\frac{(3t-\frac{3}{2})+\frac{3}{2}\frac{1}{3^{t}}}{(3t)(1-\frac{1}{3^{t+1}})}%
\leq 1,  \label{less}
\end{equation}%
since $(3t-\frac{3}{2})+\frac{3}{2}\frac{1}{3^{t}}-(3t)(1-\frac{1}{3^{t+1}})=%
\frac{1}{2}\left( 2t-3^{t+1}+3\right) \frac{1}{3^{t}}<0$ for any $t\geq 1.$
We also have%
\begin{equation}
\lim_{t\rightarrow \infty }\chi (t)=1.  \label{equal}
\end{equation}

On the other hand, using $\bar{\alpha}_{k}\geq \frac{\bar{\alpha}_{q}}{q}%
(k-q+1)$ in (\ref{lower}), for any $t$ we have
\begin{equation*}
\kappa _{t}\geq \frac{\sum\nolimits_{k=0}^{t}\frac{\bar{\alpha}_{q}(k-q+1)}{q%
}3^{k}}{\sum\nolimits_{k=0}^{t}3^{k}}\geq \frac{\bar{\alpha}_{q}}{q}\left(
\frac{\sum\nolimits_{k=0}^{t}k\cdot 3^{k}}{\sum\nolimits_{k=0}^{t}3^{k}}%
-q+1\right) ,
\end{equation*}%
that is
\begin{equation}
\kappa _{t}\geq \frac{\bar{\alpha}_{q}}{q}(\chi (t)\cdot t-q+1).
\label{k-lower}
\end{equation}

Denote%
\[
\begin{split}
\pi _{t} =\sum\limits_{\sigma ,\tau \in V_{t}}d_{t}(\sigma ,\tau ),\quad \;
&\mu _{t}=\sum\limits_{i}\sum\limits_{i\preceq \sigma ,i\preceq \tau
}d_{t}(\sigma ,\tau ), \\
\lambda _{t} =\sum\limits_{\sigma \in V_{t}}d_{t}(\sigma ,\emptyset ),%
\quad \quad &\nu _{t}=\sum\limits_{i\neq j}\sum\limits_{i\preceq \sigma
,j\preceq \tau }d_{t}(\sigma ,\tau ).
\end{split}%
\]
Then
\begin{equation*}
\pi _{t}=\mu _{t}+\lambda _{t}+\nu _{t}.
\end{equation*}%
By Lemma \ref{l:3}, we have
\begin{equation*}
\mu _{t}=3\sum\nolimits_{\sigma ^{\prime },\tau ^{\prime }\in
V_{t-1}}d_{t-1}(\sigma ^{\prime },\tau ^{\prime })=3\pi _{t-1},
\end{equation*}%
and thus%
\begin{equation}
\pi _{t}=3\pi _{t-1}+\lambda _{t}+\nu _{t}.  \label{imp3}
\end{equation}

\textbf{(1) The estimate of }$\lambda _{t}:$ Using (\ref{k-upper})-(\ref%
{less}), we have%
\begin{eqnarray}
\lambda _{t} &=&\left( \frac{\sum\nolimits_{\sigma \in V_{t}}d_{t}(\sigma
,\emptyset )}{\#V_{t}}\right) \#V_{t}  \notag \\
&=&\left( \frac{\sum\nolimits_{\sigma \in V_{t}}L_{t}(\sigma ,\emptyset )}{%
\#V_{t}}+\frac{\sum\nolimits_{\sigma \neq \emptyset }1}{\#V_{t}}\right)
\#V_{t}  \label{lambda=} \\
&=&\kappa _{t}\#V_{t}+(\#V_{t}-1)\leq \alpha ^{\ast }t(\#V_{t})+(\#V_{t}-1).
\notag
\end{eqnarray}

\textbf{(2) The estimate of} $\nu _{t}:$ For $\sigma =i\sigma ^{\prime },$ $%
\tau =j\tau ^{\prime }$ with $i\neq j,$ by Lemma \ref{l:3} we have
\begin{equation}
d_{t}(\sigma ,\tau )\leq d_{t}(i\sigma ^{\prime },i)+d_{t}(j\tau ^{\prime
},j)+d_{t}(i,j)=d_{t-1}(\sigma ^{\prime },\emptyset )+d_{t-1}(\tau ^{\prime
},\emptyset )+1.  \label{new}
\end{equation}%
Notice that
\begin{equation}
d_{t-1}(\sigma ^{\prime },\emptyset )\leq L(\sigma ^{\prime })+1\text{ for
all }\sigma ^{\prime }\in V_{t-1},  \label{new1}
\end{equation}%
since $d_{t-1}(\sigma ^{\prime },\emptyset )=L(\sigma ^{\prime })+1$ for $%
\sigma ^{\prime }\neq \emptyset $ and (\ref{new1}) is also true for $\sigma
^{\prime }=\emptyset .$

Using (\ref{k-upper})-(\ref{less}) and (\ref{new})-(\ref{new1}), we have the
following upper bound of $\nu _{t}.$%
\begin{eqnarray}
\nu _{t} &=&\sum\nolimits_{i\neq j}\sum\nolimits_{i\preceq \sigma ,j\preceq
\tau }d_{t}(\sigma ,\tau )  \notag \\
&\leq &C_{3}^{2}\sum\nolimits_{\sigma ^{\prime },\tau ^{\prime }\in
V_{t-1}}(d_{t-1}(\sigma ^{\prime },\emptyset )+d_{t-1}(\tau ^{\prime
},\emptyset )+1)  \notag \\
&\leq &3(\#V_{t-1})^{2}+6(\#V_{t-1})^{2}\frac{\sum\nolimits_{\sigma ^{\prime
}\in V_{t-1}}(L(\sigma ^{\prime })+1)}{\#V_{t-1}}  \label{nu-upper} \\
&\leq &3(\#V_{t-1})^{2}+6(\#V_{t-1})^{2}\left( \kappa _{t-1}+1\right)  \notag
\\
&\leq &9(\#V_{t-1})^{2}+6(\#V_{t-1})^{2}\cdot \alpha ^{\ast }\cdot (t-1).
\notag
\end{eqnarray}

On the other hand, for $i\neq j,$ using the Jordan curve theorem, we have
\begin{equation*}
d_{t}(i\sigma ^{\prime },j\tau ^{\prime })\geq L_{t-1}(\sigma ^{\prime
})+L_{t-1}(\tau ^{\prime }).
\end{equation*}%
Then we obtain the following lower bound of $\nu _{t}.$
\begin{eqnarray}
\nu _{t} &=&\sum\nolimits_{i\neq j}\sum\nolimits_{i\preceq \sigma ,j\preceq
\tau }d_{t}(\sigma ,\tau )  \notag \\
&\geq &C_{3}^{2}\sum\nolimits_{\sigma ^{\prime },\tau ^{\prime }\in
V_{t-1}}(L_{t-1}(\sigma ^{\prime })+L_{t-1}(\tau ^{\prime }))  \notag \\
&\geq &6\frac{\sum\nolimits_{\sigma ^{\prime }\in V_{t-1}}L_{t-1}(\sigma
^{\prime })}{(\#V_{t-1})}(\#V_{t-1})^{2}  \label{nu-lower} \\
&\geq &6\kappa _{t-1}(\#V_{t-1})^{2}  \notag \\
&\geq &6\left( \frac{\kappa _{t-1}}{t-1}\right) (\#V_{t-1})^{2}(t-1),  \notag
\end{eqnarray}%
where
\begin{equation}
\frac{\kappa _{t-1}}{t-1}\rightarrow \alpha ^{\ast }\text{ as }t\rightarrow
\infty  \label{text}
\end{equation}%
since
\begin{equation*}
\frac{\kappa _{t-1}}{t-1}=\frac{\sum\nolimits_{k=0}^{t-1}\left( \frac{\bar{%
\alpha}_{k}}{k}\right) k3^{k}}{(t-1)\sum\nolimits_{k=0}^{t-1}3^{k}}
\end{equation*}%
with $\lim\limits_{k\rightarrow \infty }\frac{\bar{\alpha}_{k}}{k}=\alpha
^{\ast }$ and $\lim\limits_{t\rightarrow \infty }\frac{\sum%
\nolimits_{k=0}^{t-1}k3^{k}}{(t-1)\sum\nolimits_{k=0}^{t-1}3^{k}}%
=\lim\limits_{t\rightarrow \infty }\chi (t-1)=1.$

\begin{proof}[Proof of Proposition 3]
\

\textbf{(i)} \textbf{Upper bound of} $\pi _{t}:$ Using (\ref{lambda=}) and (%
\ref{nu-upper})$,$ we have%
\begin{equation}
\lambda _{t}+\nu _{t}\leq \psi (t)+6(\#V_{t-1})^{2}\alpha ^{\ast }(t-1),
\label{nu_lambda}
\end{equation}%
where $\psi (t)=\alpha ^{\ast }t(\#V_{t})+(\#V_{t}-1)+9(\#V_{t-1})^{2}.$

Fix an integer $q.$ Using (\ref{imp3}) and (\ref{nu_lambda}) again and
again, for $t>q$ we have%
\[
\begin{split}
\pi _{t} &\leq 3\pi _{t-1}+\psi (t)+6\alpha ^{\ast }(t-1)(\#V_{t-1})^{2} \\
&\leq 3^{2}\pi _{t-2}+\left( \psi (t)+3\psi (t-1)\right) +6\alpha ^{\ast
}\left( (t-1)(\#V_{t-1})^{2}+18\alpha ^{\ast }(t-2)(\#V_{t-2})^{2}\right) \\
&\leq \cdots \leq 3^{t-q}\pi _{q}+\sum\limits_{k=q}^{t-1}3^{t-k-1}\psi
(k+1)+6\alpha ^{\ast }\sum\limits_{k=q}^{t-1}3^{t-k-1}k(\#V_{k})^{2}.
\end{split}
\]

We can check that
\begin{equation*}
\frac{3^{t-q}\pi _{q}+\sum\nolimits_{k=q}^{t-1}3^{t-k-1}\psi (k+1)}{%
t(\#V_{t})^{2}}\rightarrow 0\text{ as }t\rightarrow \infty .
\end{equation*}%
In fact, we only need to estimate \newline
(i) $\frac{\sum\nolimits_{k=q}^{t-1}3^{t-k-1}(k+1)(\#V_{k+1})}{t(\#V_{t})^{2}%
}=\frac{\sum\nolimits_{k=q+1}^{t}3^{t-k}k\frac{3^{k+1}-1}{2}}{t(\frac{%
3^{t+1}-1}{2})^{2}}\leq \frac{3}{2}\frac{\frac{t(t+1)}{2}3^{t}}{t(\frac{%
3^{t+1}-1}{2})^{2}}\rightarrow 0$ as $t\rightarrow \infty ;$ \newline
(ii) $\frac{\sum\nolimits_{k=q}^{t-1}3^{t-k-1}(\#V_{k+1}-1)}{t(\#V_{t})^{2}}%
\leq \frac{\sum\nolimits_{k=q}^{t}3^{t-k-1}(\#V_{k})^{2}}{t(\#V_{t})^{2}}%
\leq \frac{1}{4}\frac{3^{t}\sum\nolimits_{k=0}^{t}3^{k+1}}{t(\frac{3^{t+1}-1%
}{2})^{2}}\rightarrow 0$ as $t\rightarrow \infty .$

Therefore we obtain that%
\[
\begin{split}
\limsup_{t\rightarrow \infty }\frac{\bar{D}(t)}{t} &=\limsup_{t\rightarrow
\infty }\frac{\pi _{t}}{t(\#V_{t}-1)\#V_{t}/2} \\
&\leq 6\alpha ^{\ast }\lim_{t\rightarrow \infty }\frac{\sum%
\nolimits_{k=q}^{t-1}3^{t-k-1}k(\#V_{k})^{2}}{t(\#V_{t}-1)\#V_{t}/2} \\
&=12\alpha ^{\ast }\lim_{t\rightarrow \infty }\frac{\sum%
\nolimits_{k=q}^{t-1}3^{t-k-1}k(\#V_{k})^{2}}{t(\#V_{t})^{2}}.
\end{split}%
\]
Using Stolz theorem, we have%
\[
\begin{split}
\lim_{t\rightarrow \infty }\frac{\sum\nolimits_{k=q}^{t-1}3^{t-k-1}k(%
\#V_{k})^{2}}{t(\#V_{t})^{2}} &=\frac{1}{3}\lim_{t\rightarrow \infty }\frac{%
\sum\nolimits_{k=q}^{t-1}k(\#V_{k})^{2}/3^{k}}{t(V_{t})^{2}/3^{t}} \\
&=\frac{1}{3}\lim_{t\rightarrow \infty }\frac{(t-1)(\#V_{t-1})^{2}/3^{t-1}}{%
t(\#V_{t})^{2}/3^{t}-(t-1)(\#V_{t-1})^{2}/3^{t-1}} \\
&=\lim_{t\rightarrow \infty }\frac{(t-1)(3^{t}-1)^{2}}{6(t\cdot
3^{2t})+3^{2t+1}-6\cdot 3^{t}-2t+3} \\
&=\frac{1}{6}.
\end{split}%
\]
That means
\begin{equation}
\limsup_{t\rightarrow \infty }\frac{\bar{D}(t)}{t}\leq 2\alpha ^{\ast }.
\label{Dlimsup}
\end{equation}

\textbf{(ii) Lower bound of} $\pi _{t}:$ By (\ref{text}), suppose there
exists an integer $k_{0}$ such that $\frac{\kappa _{t-1}}{t-1}\geq (\alpha
^{\ast }-\varepsilon )$ for all $t\geq k_{0}.$ Using (\ref{imp3}) and (\ref%
{nu-lower}) we have%
\begin{eqnarray*}
\pi _{t} &\geq &3\pi _{t-1}+\nu _{t} \\
&\geq &3\pi _{t-1}+6(\alpha ^{\ast }-\varepsilon )(\#V_{t-1})^{2}(t-1) \\
&\geq &3^{2}\pi _{t-2}+6(\alpha ^{\ast }-\varepsilon )\left(
(\#V_{t-1})^{2}(t-1)+3(\#V_{t-2})^{2}(t-2)\right) \\
&\geq &\cdots \geq 6(\alpha ^{\ast }-\varepsilon
)\sum\nolimits_{k=k_{0}}^{t-1}3^{t-1-k}k(\#V_{k})^{2}.
\end{eqnarray*}%
In the same way as above, we obtain that
\begin{equation}
\liminf_{t\rightarrow \infty }\frac{\bar{D}(t)}{t}\geq 2\alpha ^{\ast }.
\label{Dliminf}
\end{equation}

It follows from (\ref{Dlimsup}) and (\ref{Dliminf}) that%
\begin{equation*}
\lim_{t\rightarrow \infty }\frac{\bar{D}(t)}{t}=2\alpha ^{\ast }.
\end{equation*}
\end{proof}

\medskip

\section{Determination of $\alpha^*$}

\subsection{Normal decomposition}

$\ $

Given a word $\sigma ,$ let $C(\sigma )$ be the set of letters appearing in $%
\sigma ,$ $\#\sigma $ the cardinality of $C(\sigma ),$ and $\sigma |_{-1}$
the last letter of $\sigma .$

For $\sigma =222113112312$, $\omega (\sigma )=4,$ we can obtain a
decomposition
\begin{equation*}
\sigma =(222)(11311)(23)(12)=\tau _{1}\tau _{2}\tau _{3}\tau _{4}
\end{equation*}%
such that $\tau _{2},\tau _{3},\tau _{4}$ contains $2$ letters and $3\tau
_{4},$ $1\tau _{3},$ $2\tau _{2}$ contains $3$ letters, where $3,$ $1$ and $%
2 $ are the last letter of $\tau _{3},\tau _{2}$ and $\tau _{1}$
respectively. In the same way, for $\sigma $ with $\omega (\sigma )=l>1,$ we
have the decomposition
\begin{equation}
\sigma =\tau _{1}\tau _{2}\cdots \tau _{l-1}\tau _{l}\text{ with }\omega
(\sigma )=l>1  \label{t-1}
\end{equation}%
satisfying%
\begin{equation}\label{tt}
\begin{split}
&\#\tau _{1} \leq 2\text{ and }\#\tau _{i}=2\text{ for }i\geq 2, \\
&|\tau _{1}| \geq 1\text{ and }|\tau _{i}|\geq 2\text{ for }i\geq 2, \\
&\{\tau _{i}|_{-1}\} =\{1,2,3\}\backslash C(\tau _{i+1}),
\end{split}%
\end{equation}
where the last one means the tail of $\tau _{i}$ with $i<l$ is \emph{%
uniquely determined} by $\tau _{i+1}.$ For $\sigma $ with $\omega (\sigma
)=1,$ we have $\#\sigma \leq 2,$ we give the decomposition
\begin{equation}
\sigma =\tau _{1}\text{ with }\omega (\sigma )=1  \label{t+1}
\end{equation}%
and (\ref{tt}) also holds. We call the decomposition (\ref{t-1}) or (\ref%
{t+1}) the \textbf{normal decomposition} if (\ref{tt}) holds. For $k\geq 3,$
let
\begin{eqnarray*}
T(k) &=&\#\{|\tau |=k\text{ with letters in }\{1,2\}:\#\tau =2\}, \\
h(k) &=&\#\{|\tau |=k\text{ with letters in }\{1,2,3\}:\tau |_{-1}=1\text{
and }\#\tau \leq 2\}, \\
M(k) &=&\#\{|\tau |=k\text{ with letters in }\{1,2\}:\tau |_{-1}=1\text{ and
}\#\tau =2\}, \\
e(k) &=&\#\{|\tau |=k\text{ with letters in }\{1,2,3\}:\#\tau \leq 2\}.
\end{eqnarray*}%
Then
\begin{eqnarray*}
T(k) &=&2^{k}-2,\text{ \quad }h(k)=2^{k}-1, \\
M(k) &=&2^{k-1}-1,\text{ }e(k)=3\cdot 2^{k}-3=3h(k).
\end{eqnarray*}%
Notice that $e(k)=\#\{|\tau |=k:\omega (\tau )=1\}$ and%
\begin{equation}
2M(k)=T(k).  \label{4}
\end{equation}

Given $k_{1}+\cdots +k_{l}=t$ with $k_{1}\geq 1$ and $k_{2},\cdots
,k_{l}\geq 2,$ consider
\begin{equation*}
W_{k_{1}\cdots k_{l}}=\#\{|\sigma |=t:\sigma =\tau _{1}\tau _{2}\cdots \tau
_{l-1}\tau _{l}\text{ are normal with }|\tau _{i}|=k_{i}\text{ for all }i\}.
\end{equation*}

\begin{lemma}
If $l\geq 3,$ then
\begin{equation}
W_{k_{1}\cdots k_{l}}=h(k_{1})\left[ (C_{2}^{1}M(k_{2}))\cdots
((C_{2}^{1}M(k_{l-1}))\right] (C_{3}^{2}T(k_{l})).\text{ }  \label{hei}
\end{equation}%
For $l=2,$ we have $W_{k_{1}k_{2}}=h(k_{1})(C_{3}^{2}T(k_{l})).$
\end{lemma}

\begin{proof}
Fix $l\geq 3.$ At first we can choose two distinct letters $i_{1}<i_{2}$
from $\{1,2,3\}$ such that $C(\tau _{l})=\{i_{1},i_{2}\},$ then the number
of choices for $\tau _{l}$ is $C_{3}^{2}T(k_{l}).$ When $\tau _{l}$ is
given, the tail of $\tau _{l-1},$ say $1,$ is uniquely determined by $\tau
_{l},$ then the number of choices for $\tau _{l-1}$ is $C_{2}^{1}M(k_{l-1})$%
. Again and again, when $\tau _{2}$ is given$,$ then the tail of $\tau _{1}$
is uniquely determined and number of choices for $\tau _{1}$ is $h(k_{1}).$
Then (\ref{hei}) follows.
\end{proof}

Then we have
\begin{equation*}
\sum\limits_{l\geq 1}\sum\limits_{\substack{ k_{1}\geq 1,k_{2},\cdots
k_{l}\geq 3  \\ k_{1}+\cdots +k_{l}=t}}W_{k_{1}\cdots
k_{l}}=e(t)+\#\{|\sigma |=t:\omega (\sigma )\geq 2\}=3^{t},
\end{equation*}%
and
\begin{equation*}
\sum\limits_{|\sigma |=t}\omega (\sigma )=e(t)+\sum\limits_{l\geq 2}\Bigg(
l\cdot \sum\limits_{\substack{ k_{1}\geq 1,\text{ }k_{2},\cdots ,k_{l}\geq 2
\\ k_{1}+\cdots +k_{l}=t}}W_{k_{1}\cdots k_{l}}\Bigg) .
\end{equation*}

\begin{lemma}
If $l\geq 2,$ then%
\begin{equation}
W_{k_{1}\cdots k_{l-1}k_{l}}=T(k_{l})W_{k_{1}\cdots k_{l-1}}\text{.}
\label{test}
\end{equation}
\end{lemma}

\begin{proof}
If $l\geq 3,$ using (\ref{4}) and (\ref{hei}), we obtain (\ref{test}). If $%
l=2,$ we have
\begin{equation}
W_{k_{1}k_{2}}=h(k_{1})(C_{3}^{2}T(k_{2}))=T(k_{2})(3h(k_{1}))=T(k_{2})(e(k_{1}))=T(k_{2})W_{k_{1}}
\label{666}
\end{equation}%
since $3h(k)=e(k).$
\end{proof}

\subsection{Proof of Proposition$\ $4}

\

For $x=\cdots x_{2}x_{1}\in \Sigma ,$ let $x|_{-k}=x_{k}x_{k-1}\cdots x_{1}.$

Given $k_{1}\geq 1$ and $k_{2},\cdots ,k_{q}\geq 2,$ consider
\[
\begin{split}
A_{k_{1}\cdots k_{q}}=\{x\in \Sigma : \;&(x1)|_{-(k_{1}+\cdots +k_{q})}=\tau
_{1}\tau _{2}\cdots \tau _{q-1}\tau _{q}\text{ is a } \\
&\text{normal decomposition with }|\tau _{i}|=k_{i}\text{ for all }i\}.
\end{split}%
\]
Since $1$ is the tail of $\tau _{q}$, we have
\begin{equation}
\mu (A_{k_{1}\cdots k_{q}})=\frac{W_{k_{1}\cdots k_{q}}/3}{3^{k_{1}+\cdots
+k_{q}-1}}=\frac{W_{k_{1}\cdots k_{q}}}{3^{k_{1}+\cdots +k_{q}}}.
\label{haha}
\end{equation}

Suppose $x1=\cdots x_{p_{n+1}}\cdots x_{p_{n}}\cdots x_{p_{1}}\cdots 1\ $and
$Y_{t}(x)=n,$ then
\begin{equation*}
(x1)|_{-t}=(x_{t}\cdots x_{p_{n}})(x_{(p_{n}-1)}\cdots x_{p_{n-1}})\cdots
(x_{p_{2}}\cdots x_{p_{1}})(x_{(p_{1}-1)}\cdots 1)
\end{equation*}%
with $p_{n+1}>t\geq p_{n}$ and
\begin{equation}
x\in A_{k_{1}k_{2}\cdots k_{n+1}}\text{ with } k_1+\cdots +k_{n+1}=t \text{ and } Y_t(x) =n  \label{haha1}
\end{equation}%
where $k_{1}=t-p_{n}+1 (\geq 1)$ and $k_{i}=p_{n-i+2}-p_{n-i+1} (\geq 2)$ for $i\geq 2.$

\begin{lemma}
\label{l: lim}Suppose $J_{n}=S_{1}+S_{2}+\cdots +S_{n}$ and $Y_{t}=\sup
\{n:J_{n}\leq t\}\ $are defined in Section 2. Then
\begin{eqnarray*}
\liminf\limits_{t\rightarrow \infty }\frac{\sum\nolimits_{|\sigma
|=t}d(\sigma ,\emptyset )}{t3^{t}} &\geq &\liminf\limits_{t\rightarrow
\infty }\frac{\sum\nolimits_{k=2}^{t-1}\mathbb{E}(Y_{t-k})\frac{2^{k}-2}{%
3^{k}}}{t}, \\
\limsup_{t\rightarrow \infty }\frac{\sum\nolimits_{|\sigma |=t}d(\sigma
,\emptyset )}{t3^{t}} &\leq &\limsup_{t\rightarrow \infty }\frac{%
\sum\nolimits_{k=2}^{t-1}\mathbb{E}(Y_{t-k})\frac{2^{k}-2}{3^{k}}}{t}.
\end{eqnarray*}
\end{lemma}

\begin{proof}
For $Y_{t^{\prime }}=\sup \{n:J_{n}\leq t^{\prime }\}$, using (\ref{haha1})
we have%
\begin{equation*}
\mathbb{E}(Y_{t^{\prime }})=\sum_{q}\sum\limits_{\substack{ k_{1}+\cdots
+k_{q}=t^{\prime } \\ k_{1}\geq 1,k_{2},\cdots ,k_{q}\geq 2}}(q-1)\mu
(A_{k_{1}\cdots k_{q}}).
\end{equation*}%
Using (\ref{test}) and (\ref{haha}), we obtain that
\begin{eqnarray*}
&&\frac{\sum\nolimits_{|\sigma |=t}\omega (\sigma )}{3^{t}} \\
&=&\frac{e(t)}{3^{t}}+\sum\limits_{k=2}^{t-1}\sum\limits_{l\geq 2}\Bigg(
l\cdot \frac{T(k)}{3^{k}}\sum\limits_{\substack{ k_{1}+\cdots +k_{l-1}=t-k
\\ k_{1}\geq 1,k_{2},\cdots ,k_{l-1}\geq 2}}\mu (A_{k_{1}\cdots
k_{l-1}})\Bigg)
\end{eqnarray*}%
where $\frac{e(t)}{3^{t}}\rightarrow 0$. We notice that%
\begin{eqnarray*}
&&\sum\limits_{k=2}^{t-1}\sum\limits_{l\geq 2}\sum\limits_{\substack{ %
k_{1}+\cdots +k_{l-1}=t-k \\ k_{1}\geq 1,k_{2},\cdots ,k_{l-1}\geq 2}}l\cdot
\frac{T(k)}{3^{k}}\mu (A_{k_{1}\cdots k_{l-1}}) \\
&\leq &2\sum\limits_{k\geq 2}\frac{T(k)}{3^{k}}+\sum\limits_{k=2}^{t-1}\sum%
\limits_{l\geq 2}\sum\limits_{\substack{ k_{1}+\cdots +k_{l-1}=t-k \\ %
k_{1}\geq 1,k_{2},\cdots ,k_{l-1}\geq 2}}(l-2)\frac{T(k)}{3^{k}}\mu
(A_{k_{1}\cdots k_{l-1}}) \\
&\leq &2+\sum\limits_{k=2}^{t-1}\mathbb{E}(Y_{t-k})\frac{T(k)}{3^{k}}.
\end{eqnarray*}%
On the other hand, we have%
\begin{eqnarray*}
&&\sum\limits_{k=2}^{t-1}\sum\limits_{l\geq 2}\sum\limits_{\substack{ %
k_{1}+\cdots +k_{l-1}=t-k \\ k_{1}\geq 1,k_{2},\cdots ,k_{l-1}\geq 2}}l\cdot
\frac{T(k)}{3^{k}}\mu (A_{k_{1}\cdots k_{l-1}}) \\
&\geq &\sum\limits_{k=2}^{t-1}\sum\limits_{l\geq 2}\sum\limits_{\substack{ %
k_{1}+\cdots +k_{l-1}=t-k \\ k_{1}\geq 1,k_{2},\cdots ,k_{l-1}\geq 2}}(l-2)%
\frac{T(k)}{3^{k}}\mu (A_{k_{1}\cdots k_{l-1}}) \\
&\geq &\sum\limits_{k=2}^{t-1}\mathbb{E}(Y_{t-k})\frac{T(k)}{3^{k}}.
\end{eqnarray*}%
Notice that $\frac{e(t)}{3^{t}}\rightarrow 0,$ then the lemma follows.
\end{proof}

\begin{proof}[Proof of Proposition 4]
Notice that $\frac{\mathbb{E}(Y_{i})}{i}\rightarrow \frac{2}{9}$ as $%
i\rightarrow \infty $ and
\begin{equation*}
\frac{\sum\nolimits_{k=2}^{t-1}\left( \mathbb{(}t-k\mathbb{)\cdot }\frac{%
2^{k}-2}{3^{k}}\right) }{t}\rightarrow 1\text{ as }t\rightarrow \infty ,
\end{equation*}%
using Lemma \ref{l: lim} we obtain that%
\begin{eqnarray*}
\alpha ^{\ast }=\lim_{t\rightarrow \infty }\frac{\sum\nolimits_{|\sigma
|=t}\omega (\sigma )}{t3^{t}} &=&\lim_{t\rightarrow \infty }\frac{%
\sum\nolimits_{k=2}^{t-1}\mathbb{E}(Y_{t-k})\frac{2^{k}-2}{3^{k}}}{t} \\
&=&\lim_{t\rightarrow \infty }\frac{\sum\nolimits_{k=2}^{t-1}\frac{\mathbb{E}%
(Y_{t-k})}{t-k}(t-k)\frac{2^{k}-2}{3^{k}}}{t} \\
&=&\lim_{i\rightarrow \infty }\frac{\mathbb{E}(Y_{i})}{i}=\frac{2}{9}.
\end{eqnarray*}%
\bigskip
\end{proof}

\end{document}